\documentclass{amsart}
\usepackage{amsmath,amssymb,amsthm,latexsym}
\usepackage{graphicx}
\usepackage{stmaryrd}		% \llbracket and \rrbracket
\usepackage{xcolor}
\usepackage{hyperref}

\graphicspath{{figures/}}
\newtheorem{theorem}{Theorem}[section]
\newtheorem{corollary}[theorem]{Corollary}
\newtheorem{lemma}[theorem]{Lemma}
\newtheorem{proposition}[theorem]{Proposition}

\theoremstyle{definition}

\newtheorem{remark}[theorem]{Remark}

\newcommand{\F}{\mathbb F}

\DeclareMathOperator{\Span}{span}
\newcommand{\comp}{{\rm comp}}

\author{Boris Adamczewski}
\address{Universit\'{e} de Lyon,
Universit\'e Claude Bernard Lyon 1, CNRS UMR 5208, Institut Camille Jordan, F-69622 Villeurbanne Cedex, France}
\email{boris.adamczewski@math.cnrs.fr}

\author{Reem Yassawi}
\address{
 Universit\'{e} de Lyon, Universit\'e Claude Bernard Lyon 1, CNRS UMR 5208, Institut Camille Jordan, F-69622 Villeurbanne Cedex, France}
\email{yassawi@math.univ-lyon1.fr}

\title{A note on Christol's theorem}
\date{}

\thanks{This project has received funding from the European Research Council (ERC) under the European Union's Horizon 2020 
research and innovation programme under the Grant Agreement No 648132. }

\begin{document}

%%%%%%%%%%
\begin{abstract} 
Christol's theorem characterises algebraic power series over  finite fields in terms of finite automata.
In a recent article, Bridy develops a new proof of Christol's theorem by Speyer, to obtain a tight quantitative version, 
that is, to bound the size of the corresponding automaton in terms of the height and degree of the power series, 
as well as the genus of the  curve associated with the minimal polynomial of the power series.  
Speyer's proof, and Bridy's development, both take place in the setting of algebraic geometry, in particular by considering  
K\"ahler differentials of the function field of the curve. 
In this note we show how an elementary approach, based on diagonals of bivariate rational functions, provides essentially the same bounds.  
 \end{abstract}
%%%%%%%%%%%%%%%%%%%%%%%%%%%%%%%%%%%%%%%%%%%%%%%%%%%%%%%%
%%%%%%%%%%%%%%%%%%%%

\maketitle

\section{Introduction}

The title of this paper refers to the following classical result of Christol.  

\begin{theorem}[Christol]
Let $q$ be a power of a prime number $p$ and let 
$f(x)=\sum_{n=0}^{\infty} a_nx^n\in \mathbb F_q[[x]]$. Then $f(x)$ is algebraic over 
$\mathbb F_q(x)$ if and 
only if the sequence ${\bf a}=(a_n)_{n\geq 0}$ is $q$-automatic. 
\end{theorem}

Here, an infinite sequence ${\bf a}=(a_n)_{n\geq 0}$ is {\em $q$-automatic} if $a_n$ is a finite-state function of the base-$q$ 
expansion of $n$. This means that there exists a deterministic finite automaton with output  taking the base-$q$ 
expansion of $n$ as input, and producing the symbol $a_n$ as output. For a formal definition, we refer the reader to \cite[Chapter 5]{Allouche-Shallit}. 
Christol's theorem is easy to prove, but nevertheless deep, in the sense that it provides an intimate connection 
between two apparently unrelated areas.   
On the one hand, algebraic power series with coefficients over finite fields are fundamental for 
arithmetic in positive characteristic while, on the other hand,  
finite automata are fundamental for computer science.  
On each side, there is a natural way to measure the complexity of the corresponding objects. 
The complexity of an algebraic power series $f$ is measured by its degree $d$ and its height $h$.
Here, the degree of $f$ is the degree of the field extension $[\mathbb F_q(x)(f(x)):\mathbb F_q(x)]$, while the height of $f$ 
is the minimal degree (in $x$) of a nonzero polynomial $P(x,y)\in\mathbb F_q[x,y]$ such that $P(x,f(x))=0$.  
With a more geometric flavor, one can also add the genus $g$ 
of the curve associated with the minimal polynomial of $f$. 
The complexity of a $q$-automatic sequence $\bf a$, denoted by  $\comp_q({\bf a})$, is 
measured  by the number of states in a
minimal  finite automaton generating $\bf a$ in {\em reverse} reading, 
by which we mean that the input $n$ is read starting from the least significant digit.  
In Section \ref{direct}, we also discuss bounds in {\em direct} reading.   
By a result of Eilenberg \cite{Eilenberg}, a sequence 
${\bf a}$ is $q$-automatic if and only if its \emph{$q$-kernel} 
$$
\ker_q({\bf a})=\left\{ (a_{q^rn+j})_{n\geq 0} : r\geq 0, 0\leq j<q^r\right\} 
$$
is a finite set. By \cite[Corollary 4.1.9 and Theorem 6.6.2]{Allouche-Shallit}, we have  
\begin{equation}\label{comp=ker}\comp_q({\bf a})= \vert \ker_q({\bf a})\vert.\end{equation} 
Thus we can bound the complexity of 
${\bf a}$ 
 by bounding  $|\ker_q({\bf a})|.$

We are interested here in the interplay between these two notions of complexity.   
If a sequence ${\bf a}$ is generated by a $q$-automaton with at most $m$ states, 
it is not difficult to show that the associated power series $f$ has degree at most 
$q^m-1$ and height at most $mq^{m+1}$ (see, for instance, \cite[Proposition 2.13]{Bridy}).  
Furthermore, these bounds cannot be significantly improved in general.  

Bounds in the other direction are more challenging. 
If $f$ is algebraic of degree $d$, the power series $f,f^q,f^{q^2},\ldots,f^{q^d}$ are linearly 
dependent over $\mathbb F_q(x)$ and thus there exist polynomials 
$A_0(x),\ldots,A_d(x)\in\mathbb F_q[x]$, not all zero, such that 
\begin{equation}\label{eq: ore}
A_0(x)f(x)+\cdots +A_d(x)f(x)^{q^d}=0 \, .
\end{equation}
Furthermore, it is possible to ensure that $A_0(x)\not=0$. 
Such a relation is called an \emph{Ore relation} for $f$.  
The standard proof of Christol's theorem, which dates back to \cite{CKMR}, is based on 
such a relation, and  is the proof given in  Allouche and Shallit's book \cite{Allouche-Shallit}.    
The arguments are effective and not particularly hard to quantify. Explicit bounds for 
$\comp_q({\bf a})$ can be easily extracted from \cite{H,HII,AB-1},    
where the authors work in a more general framework (arbitrary base fields of characteristic $p$ 
and power series in several variables).   
The authors of \cite{FKM} also obtain a quantitative version of Christol's theorem 
using similar techniques. The common feature of all these bounds is that they have a doubly exponential 
nature\footnote{Sharif and Woodcock \cite{SW} apparently realise that they can obtain effective bounds, 
though they do not spell them out, and here again, they would be doubly exponential.}, 
that is, they are of the form $q^{cq^{k}}$, where $c$ and $k$ are polynomial functions of $d$ and $h$.  
For, one can derive from \eqref{eq: ore} that  
$$
\comp_q({\bf a}) \leq q^{d(2H+1)} \, ,
$$
where $H:=\max\{\deg A_i(x) : 0\leq i \leq d\}$ denotes the \emph{height} of the Ore relation \eqref{eq: ore}. 
The double exponentiation appears because the upper bounds for $H$ are  already exponential in $q$. For instance, the bound  
$H\leq 2hq^d$ can be derived from \cite{HII}; see also \cite{AB-1}. 

In contrast, when $f$ is a rational function (i.e. $d=1$),  
one can easily obtain the bound $q^{h+1}$ and thus get rid of the double exponential. 
This suggests that the previous bounds are artificially large. 
In a recent paper, 
Bridy \cite{Bridy}  drastically improves on these doubly exponential bounds, confirming this guess.  
More precisely, he obtains the following essentially sharp bound: 
\begin{equation}\label{eq: bridy}
\comp_q({\bf a}) \leq (1+o(1))q^{h+d+g-1} \,,
\end{equation}
 where the $o(1)$ terms tends to $0$ for large values of any of $q,h,d$, or $g$. 
By Riemann's inequality, which gives $g\leq (h-1)(d-1)$, one can deduce that 
\begin{equation}\label{eq: bridy2}
\comp_q({\bf a}) \leq (1+o(1))q^{hd} \,.
\end{equation}
Bridy's approach is based on a new proof of Christol's theorem  in the context of algebraic geometry, due to 
Speyer \cite{Speyer-2010}.  Speyer's argument is elegant, connecting finite automata with the geometry of curves. 
Furthermore, Bridy's bound \eqref{eq: bridy} shows that this approach is also very efficient. 
However, the price to pay is that some classical background from algebraic geometry is needed: 
the Riemann-Roch theorem, existence and basic properties of the 
\emph{Cartier} operator acting on the space 
 of K\"ahler differentials of the function field associated with $f$, along with asymptotic bounds for the Landau function.  

Here, we come back to Christol's original argument \cite{Christol-1979}, 
which appears slightly before \cite{CKMR}. It is based on a result of Furstenberg \cite{Furstenberg} 
showing that any algebraic power series in $\mathbb F_q[[x]]$ can be expressed as the diagonal of  
rational power series in two variables. Using diagonals avoids the use of an Ore relation, which is the culprit behind the double 
exponentiation.  In the end, we obtain simply exponential bounds similar to those of Bridy, though slightly weaker.  
For instance, we can show that 
\begin{equation}\label{eq: our}
\comp_q({\bf a}) \leq (1+ o(1))q^{(h+1)d+1} \,,
\end{equation}
where the $o(1)$ terms tends to $0$ for large values of any of $q,d$, or $h$. 
Furstenberg's theorem has been studied and applied by a number of authors, for  example by Denef and Lipshitz \cite{DL},  the first author and Bell [AB13],
the second author and Rowland \cite{Rowland-Yassawi-2015} and Bostan, Caruso, Christol, and Dumas in \cite{Bostan-Caruso-Christol-Dumas}.

The purpose of this note  is to publicize  that  using diagonals to prove Christol's theorem gives us essentially the same bounds as the theoretically more demanding approach using the Riemann Roch theorem. Indeed, diagonals have already been used, by the first author and Bell, to obtain a singly exponential bound [AB13, Theorem 7.1]; here we simply push this technique to optimise the bounds.
Also, we point out that the methods that we describe here can also be  applied to the study of algebraic functions of several variables, Hadamard products, and reduction modulo prime powers of  diagonals of multivariate rational functions,  including transcendental ones, with little extra theoretical cost.

%%%%%%%%%%%%%%%%%%%%%%%%%%%%%%%%%%%%
\section{Cartier operators and diagonals} 

We recall here some definitions and basic results about Cartier operators and diagonals. 

\subsection{One-dimensional Cartier operators}
Let $f(x)=\sum_{n\geq 0}a_nx^n\in \F_q[[x]]$.
 For every natural number $i$, $0\leq i \leq q-1$, we define $\Lambda_i$ as  
 the $\mathbb F_q$-linear operator acting on $\mathbb F_q[[x]]$ by: 
$$
 \Lambda_i \left(f(x)\right) =  \sum_{n=0}^{\infty} a_{nq+ i}x^n \,.
$$
We  let $\Omega_1$ denote the monoid generated by these operators under  composition. 
In the framework of Christol's theorem, the operators $\Lambda_i$ are usually called \emph{Cartier operators} 
(see, for instance, \cite{Allouche-Shallit}). 
They are the tools with which one drops down to the elements of $\ker_q({\bf a})$.
Indeed, we can rephrase Equation (\ref{comp=ker}) as  
\begin{equation}\label{eq: Eil}
  \comp_q({\bf a})=
\vert \Omega_1(f)\vert
\end{equation}
where $\Omega_1(f)$ is the orbit of $f(x)$ under the action of $\Omega_1$.
 All known proofs of the sufficiency direction of Christol's theorem have the same blueprint. One finds a finite 
set that contains $f(x)$, and which is invariant under the action of the Cartier operators. 
Usually, this set is an $\mathbb F_q$-vector space, say $V$, and one just has to prove that it is finite dimensional. 
However,  in order to obtain finer quantitative results, it will be convenient to  trace the orbit $\Omega_1(f)$ 
more closely inside $V$, and in doing so, we will lose the vector space structure.

Let us briefly recall why the $\Lambda_i$'s are referred to as the Cartier operators. 
To make the connection with the \emph{real} Cartier operator $\mathcal C$, 
which comes from algebraic geometry, we redefine 
our operators as follows. For every natural number 
$i$, $0\leq i \leq p-1$, we let $\Lambda'_i$ be  
 the $\mathbb F_q$-linear operator acting on $\mathbb F_q[[x]]$ by: 
$$
 \Lambda'_i \left(\sum_{n=0}^{\infty} a_nx^n\right) =  \sum_{n=0}^{\infty} a_{np+ i}^{1/p}x^n \,.
$$
If $q=p^r$, we retrieve the operator $\Lambda_i$  as a composition of $r$  operators $\Lambda'_j$. 
With this notation,   
$\Lambda'_{p-1}$ is reminiscent of $\mathcal C$. 
Let $f(x)\in\mathbb F_q[[x]]$ be  an algebraic power series and let $X$ denote the 
smooth projective algebraic curve, obtained after the normalization of the projective closure of the affine plane curve 
defined by the minimal polynomial of $f$.  Let
$\mathbb F_q(X)$ denote the function field associated with $X$, and let   $\Omega_{\mathbb F_q(X)/\mathbb F_q}$ denote the one-dimensional $\mathbb F_q(X)$-vector space
of K\"ahler 
differentials of $\mathbb F_q(X)$.  Choosing $x$ to be a \emph{separating} variable,  i.e. $x\not\in  \mathbb F_q(X)^{p}$, we can define   $\mathcal C$ as 
\begin{equation}\label{eq: truecartier}
\mathcal C(f(x)dx) = \Lambda_{p-1}'(f(x))dx \,,
\end{equation}
so that $\mathcal C$ acts on differentials exactly as $\Lambda_{p-1}$ acts on power series.  
 In Speyer's proof of Christol's theorem, a finite dimensional $\mathbb F_q$-vector space, containing 
 $f$ and invariant under 
 $\Omega_1$,  is obtained first by finding an effective divisor $D$ on $X$
 such that 
$$
\Omega(D) := \left\{\omega\in \Omega_{\mathbb F_q(X)/\mathbb F_q}\setminus \{0\} : (\omega)+D\geq 0\right\}\cup \{0\} \,
$$ 
is invariant under $\mathcal C$ and also the {\em twisted} Cartier operators, which play the role of the other operators $\Lambda_i$, and then by using \eqref{eq: truecartier}. 
The fact that $\Omega(D)$ has finite dimension over $\mathbb F_q$ is 
a direct consequence of the Riemann-Roch theorem. Bridy finds the best choice for the divisor $D$; he also has to trace 
the orbit of $f(x)dx$ more precisely  under $\mathcal C$ and the twisted Cartier operators  
inside $\Omega(D)$. 

%%%%
\subsection{Two-dimensional Cartier operators}
The definition of the Cartier operators $\Lambda_i$ naturally extends to power series in an arbitrary number of 
variables. We recall here the two-dimensional case.  
Given $(i,j) \in  \{0, 1, \dots, q-1\}^2$, we let  $\Lambda_{i,j}$ denote 
the $\mathbb F_q$-linear operator acting on $\mathbb F_q[[x,y]]$ by: 
\begin{equation}
	 \Lambda_{i,j} \left(\sum_{n,m\geq 0} a_{n,m}x^ny^m\right) =  \sum_{n,m\geq 0} a_{nq+ i,mq+j}x^ny^m \,.
\end{equation}
Analogous to the notation in one dimension, we  let $\Omega_2$ denote the monoid generated by the two-dimensional Cartier operators under  composition. 
An elementary but fundamental property of these operators is that 
\begin{equation}\label{eq: gq}
\Lambda_{i,j}(fg^q)=\Lambda_{i,j}(f)g, 
\end{equation}
for all $f,g\in\F_q[[x,y]]$.
Notice that this useful equality also allows one to uniquely extend $\Lambda_{i,j}$ 
to the field of fractions of $\F_q[[x,y]]$ by setting 
$\Lambda(f/g):=\Lambda(fg^{q-1})/g$.  
Furthermore, if $P\in\mathbb F_q[x,y]$ then 
\begin{equation}\label{eq: pol0}
	 \deg_{x}\Lambda_{i,j} (P)\leq   \deg_{x}(P)/q\, \mbox{ and } \deg_{y}\Lambda_{i,j} (P)\leq   \deg_{y}(P)/q.
\end{equation}

%%%%
\subsection{Diagonals}

The {\em diagonal} of a power series $f(x,y)=\sum_{n,m\geq 0}a_{n,m}x^ny^m\in \mathbb F_q[[x,y]]$ is defined by 
$$
\Delta(f)(x) = \sum_{n=0}^{\infty}a_{n,n}x^n \,.
$$
It is straightforward to check that   
\begin{equation}\label{eq: commute} 
\Lambda_{i}(\Delta (f)) = \Delta(\Lambda_{i,i}(f)) \,,
\end{equation}
for all $f(x,y)\in\mathbb F_q[[x,y]]$. 
Furstenberg \cite{Furstenberg} proved that any algebraic power series in $\mathbb F_q[[x]]$ is the 
diagonal of a bivariate rational power series. His proof is based on the 
following key formula  \cite[Proposition 2]{Furstenberg}. 

\begin{lemma}[Furstenberg]\label{fursternberg-special}
Let $\mathbb K$ be a field and let $P(x,y)\in \mathbb K[x,y]$. Let $f(x)\in \mathbb K[[x]]$ be a root of $P(x,y)$. If $f(0)=0$ and $\frac{\partial P}{\partial y}(0,f(0))\neq 0$,  then 
\begin{align}f(x)=\Delta \left (\frac{y \frac{\partial P}{\partial y}(xy,y)} {y^{-1}P(xy,y)}\right )\, \cdot\label{furstenberg-formula}\end{align}
\end{lemma}

The proof of this lemma is both easy and elementary,
though it took Furstenberg's insight to find the rational function in the right hand side of  \eqref{furstenberg-formula}.  
Though the formula is valid over any field, 
it has its roots firmly planted in residue theory, where we can express an algebraic function as an integral of a certain rational function. Precisely, suppose that $P(0,0)=0$ and that $0$ is an isolated root of $P(0,y)=0$. These conditions guarantee 
that there  is a unique power series $y=f(x)$ that converges close to the origin, and  satisfying both $P(x,f(x))=0$ and $f(0)=0$. By Cauchy's  generalised residue theorem,  one can express
\begin{equation}\label{eq: res1}
 f(x) = \frac{1}{2\pi i}\int_{\gamma} \frac{y\frac{\partial P}{\partial y}(x,y)}{P(x,y)}dy 
 \end{equation}
for $\gamma$ and $|z|$ sufficiently small. 
On the other hand, if $g(x,y)\in\mathbb C[[x,y]]$, then $\Delta(g)$ is also a residue that can be simply expressed as 
\begin{equation}\label{eq: res2}
\Delta(g)= \frac{1}{2\pi i} \int_\gamma g(x/w,w) )\frac{dw}{w} \,\cdot
 \end{equation}
Now \eqref{eq: res1} and \eqref{eq: res2} allow us to deduce  \eqref{furstenberg-formula}.

%%%%%%%%%%%%%%%%%%%%%%%%%%
\section{The smooth case}

In this section, we first describe our strategy in the smooth case and show how it easily leads 
to simply exponential bounds. By the smooth case, we mean that 
 $f(x)=\sum_{n\geq 0} a_n x^n  \in x\F_q[[x]]$ is an algebraic power series whose minimal polynomial $P(x,y)$ 
satisfies Furstenberg's condition: $\frac{\partial P}{\partial y}(0,0)\neq 0$. This conditions ensures  
that the plane algebraic curve associated with $P$ is nonsingular at the origin.

 By Furstenberg's formula,  
there exists an explicit rational power series $P/Q\in\mathbb F_q(x,y)$ such that 
$f=\Delta(P/Q)$. Let $m$ be an upper bound for the degree in $x$ and $y$ of the polynomials $P$ and $Q$. 
Then 
$$
W := \left\{A/Q : \max(\deg_x(A),\deg_y(A))\leq m \right\} 
$$
is a $\mathbb F_q$-vector space of dimension $(m+1)^2$. 
Furthermore, $W$ is invariant under $\Omega_2$. Indeed, $\Lambda_{i,j}(A/Q)=\Lambda_{i,j}(AQ^{q-1})/Q$ 
and it follows from \eqref{eq: pol0} that 
$$
\deg_{x}(\Lambda_{i,j}(AQ^{q-1}))\leq m
\mbox{ and }\deg_{y}(\Lambda_{i,j}(AQ^{q-1}))\leq m
\,.$$ 
Thus we  deduce from \eqref{eq: commute} that $\Omega_1(f)\subset \Delta(W)$. 
Since $\Delta$ is linear, $\Delta(W)$ is an $\mathbb F_q$-vector space of dimension at most $(m+1)^2$. 
Furthermore, if $f$ has degree $d$ and height $h$, Furstenberg's formula shows that  
we can choose $m=h+d$.  Then we get from Equality \eqref{eq: Eil} that  
$$
\vert \Omega_1(f)\vert\leq q^{(h+d+1)^2} \,,
$$
which is already an acceptable simply exponential bound.

In the next theorem, we refine the previous argument to obtain  
a bound that is only very slightly weaker than the one given by Bridy in \eqref{eq: bridy2}.

\begin{theorem}\label{smooth-case}
Let $f(x)=\sum_{n\geq 0} a_n x^n \in x\F_q[[x]]$ be an algebraic power series of degree $d$ and height $h$, and 
let $P(x,y)\in \F_q[x,y]$ denote the minimal polynomial of $f$. If $\frac{\partial P}{\partial y}(0,0)\neq 0$,
then   $$ \comp_q({\bf a})
\leq 1+ q^{(h+1)d}.$$ 
\end{theorem}

\begin{proof}
Let $P(x,y)= \sum_{i=0}^{d}A_i(x)y^i$. By assumption, we have $A_0(0)=0$ and $A_1(0))\not=0$. This implies that $y$ divides  $P(xy,y)$.  By Lemma
\ref{fursternberg-special},  we have 
\begin{equation}\label{furst} f(x)=
\Delta \left( \frac{y\frac{ \partial P }{\partial y}(xy,y)}
      {y^{-1}P(xy,y)}\right)
      .\end{equation}
Consider the three $\F_q$-vector spaces
\begin{eqnarray}\label{nested-spaces}
 U &:=&\Span_{\mathbb F_q}\left \{ (xy)^{i}y^{j}: 0\leq i \leq qh, 0\leq j \leq qd-1 \right\}, \\
 \nonumber     V &:=&\Span_{\mathbb F_q}\left \{(xy)^{i}y^{j}: 0\leq i \leq h, 0\leq j \leq d-1 \right\}, \\
\nonumber  W &:=&\Span_{\mathbb F_q} \left \{ \frac{  (xy)^{i}y^{j} }{y^{-1}P(xy,y)}: 0\leq i \leq h, 0\leq j \leq d-1  \right \}\, .
\end{eqnarray}
Let $\ell \in \{ 0,\ldots , q-1\}$.  For any $u\in U$,  note that $\Lambda_{\ell,\ell}(u) \in V$.
We claim that $\Lambda_{\ell}(f) \in \Delta(W)$. 
 For, using Properties \eqref{eq: commute} and \eqref{eq: gq}, we  have
 
\begin{eqnarray*}
  \Lambda_{\ell}(f) 
&=&    \Lambda_{\ell} \Delta  \left (
\frac{y    \frac{ \partial P }{\partial y}(xy,y) \left(y^{-1}P(xy,y)\right)^{q-1}   }{\left( y^{-1}P(xy,y)\right)^q}
\right ) \\
&=& 
 \frac{ \Delta
 \left( \Lambda_{\ell,\ell} \left (y        \frac{ \partial P }{\partial y}(xy,y)        \left( y^{-1}P(xy,y)\right)^{q-1}  \right) \right)
}{y^{-1}P(xy,y)}\, \cdot
 \end{eqnarray*}
Now notice that the polynomial 
$y  \frac{ \partial P }{\partial y}(xy,y) \left( y^{-1}P(xy,y)\right)^{q-1}$ is an $\F_q$-linear combination of 
monomials in  the set $\{(xy)^{i}y^{j}: 0\leq i \leq qh, 2-q\leq j \leq (d-1)q+1 \}$.  Notice also that if $j\not\equiv 0 \mod q$, 
then  $\Lambda_{\ell, \ell} (   (xy)^{i}y^{j}         )=0$ for any $\ell$. This implies that 
the image of this last set under $\Lambda_{\ell,\ell}$ is the same as that of $U$, so
 \begin{eqnarray*}     \Lambda_{\ell,\ell}    \left (y        \frac{ \partial P }{\partial y}(xy,y)        \left( y^{-1}P(xy,y)\right)^{q-1}  \right)  & \in &    
   \Lambda_{\ell, \ell} \left(  U  \right)    \subset V,
    \end{eqnarray*}
    and this proves our claim that $\Lambda_{\ell}(f) \in \Delta(W)$. 
  The same reasoning  shows that 
  $\Delta (W)$ is invariant under $\Omega_1$.  Since $\Lambda_\ell(f)\in \Delta(W)$ for all $\ell\in\{0,\ldots,q-1\}$, 
  it follows that $\Omega_1(f) \setminus \{ f\}$ is a subet of  
\begin{equation*}
 \Delta(W) =    \Span_{\mathbb F_q}  \left \{ \Delta\left( \frac{  (xy)^{i}y^j }{y^{-1}P(xy,y)}\right) : 0\leq i \leq h, 0\leq i \leq d-1 \right\}\,,
\end{equation*} 
     which has dimension at most $(h+1)d$ over $\mathbb F_q$. 
 The result follows.
 \end{proof}

%%%%%%%%%%%%%%%%%%%%%%%%
\section{The general case} 

The aim of this section is to remove the assumption needed in Furstenberg's formula, 
proving the following general bound. 

 \begin{theorem}\label{bounding-f-kernel}
 Let $f(x)=\sum_{n\geq 0} a_n x^n \in \F_q[[x]]$ be an algebraic power series of degree $d$ and height $h$.  
 Then 
 $$
 \comp_q({\bf a})\leq (1+ o(1))q^{(h+1)d+1}\,,
 $$
 where the $o(1)$ term tends to 0 for 
 large values of any of $q$, $h$, or $d$. 
 \end{theorem}

\noindent{\bf Notation.} Throughout this section, we let
$f(x) = \sum_{n\geq0} a_n x^n\in \mathbb F_q[[x]]$ denote an algebraic power series 
of degree $d$ and height $h$, and we let
$P(x,y)$ denote its minimal polynomial. 
 Also we fix $r$ to be  the order at $0$ of the resultant of $P(x,y)$ 
and $ \frac{\partial P}{\partial y}(x,y)$. By the 
 determinantal formula for the resultant, we have  
\begin{equation} \label{determinant} r\leq h(2d-1). \end{equation} 

 We define
 \begin{eqnarray} \label{f-tail-expression}
 \nonumber V_r(x)&:=&\sum_{n=0}^r a_nx^n \, , \\
  f^{(r)}(x)&:=& x^{-r}( f(x)-V_r(x))\in x\mathbb F_q[[x]] \, , \\
  \nonumber M_r(x,y)&:=& V_r(x) + x^ry  \, ,\mbox{ and } \\
\nonumber  Q_r(x,y)&:=&P(x,M_r(x,y)) \in \mathbb F_q[x;y].
\end{eqnarray}
Notice that
\begin{equation}\label{eq: Qr}
f^{(r)}(0)=0 \,\,\mbox{  and }\,\, Q_r(x,f^{(r)}(x))=0 \,.
\end{equation}    

The following elementary argument shows that $f^{(r)}$ is the only power series root of $Q_r$ with no constant term.  
 The proof follows the argument given in \cite[proof of Lemma 6.2]{Adam-Bell-2013}; we include it for the sake of 
completeness.

 \begin{lemma}\label{separation-lemma}  
 
There exists a nonnegative integer $s$ such that $x^{-s}Q_r(x,y)$ is a polynomial satisfying 
 the condition of Furstenberg's formula, that is 
$\frac{\partial (x^{-s}Q_r)}{\partial y}(0,0)\neq 0$, so that
$$
f^{(r)}(x) = \Delta\left(\frac{y \frac{\partial Q_r}{\partial y}(xy,y)} {y^{-1}Q_r(xy,y)}\right ) \,\cdot
$$
 \end{lemma}
 
 \begin{proof}
 Let $S(x)$ denote the resultant of $P(x,y)=\sum_{i=0}^dA_i(x)y^{i}$ and $ \frac{\partial P}{\partial y}(x,y)$, 
 so that $S(x)=x^rT(x)$ with $T(0)\not=0$. 
 To simplify notation, let  $V(x):=V_r(x)$, $g(x):= f^{(r)}(x)$, and $Q:=Q_r$.    
Setting
$$
B_i(x) := \frac{1}{i!} \frac{\partial^{i} P}{\partial y^{i}}(x,V(x)) \,,
$$
we get that
$$
Q(x,y)= \sum_{i=0}^dB_i(x)x^{ri}y^{i} \,. 
$$
Note that 
\begin{eqnarray*}
P(x,V(x))&=&P(x,V(x))-P(x,f(x))\\
&=&(V(x)-f(x))C(x)\\
&=& -x^rg(x)C(x) \,,
\end{eqnarray*}
with $C(x)=\sum_{i=1}^dA_i(x)\left(\sum_{k=0}^{i-1}V(x)^kf(x)^{i-k-1}\right)\in\mathbb F_q[[x]]$. Thus 
$x^{r+1}$ divides $P(x,V(x))$, as $g(0)=0$. 
On the other hand, since $S$ is the resultant of $P$ and $ \frac{\partial P}{\partial y}$, 
there exist two polynomials $A(x,y)$ and $B(x,y)$ such that 
$$
x^rT(x)=A(x,y)P(x,y) + B(x,y)\frac{\partial P}{\partial y}(x,y) \,.
$$
It follows that 
$$
B(x,V(x))\frac{\partial P}{\partial y}(x,V(x)) = x^rT(x) -A(x,V(x))P(x,V(x))\,,
$$
which implies that $\nu\leq r$, where  $\nu$ denotes the order at $0$ of $B_1(x)=\frac{\partial P}{\partial y}(x,V(x))$. 
Since the order of $B_i(x)x^{ri} \geq 2r$ for all $i$ such that $2\leq i\leq d$, 
we obtain that the order of $B_0(x)$ at $0$ is at least equal to $s:=\nu+r$. It follows that 
$x^{-s}Q(x,y)$ is a polynomial such that 
$$
\frac{\partial (x^{-s}Q)}{\partial y}(x,0) = B_1(x)x^{-\nu} 
$$
and thus $\frac{\partial (x^{-s}Q)}{\partial y}(0,0)\not=0$, as desired. 
By \eqref{eq: Qr}, we can apply Furstenberg's Lemma to $f^{(r)}$ and $x^{-s}Q(x,y)$, which 
provides the expected formula. 
 \end{proof}

The idea behind proving Theorem \ref{bounding-f-kernel}  is to find successively shrinking 
vector spaces to which most of $\Omega_1(f)$ belongs.  Iterating the following two lemmas will allow us 
to achieve this shrinkage. In the lemmas that follow we continue with the notation given in (\ref{f-tail-expression}).

\begin{lemma}\label{shrinkage}  
Let $\ell\in\{0,\ldots,q-1\}$, let $\alpha$ be a rational number with $0\leq \alpha\leq r$,  and let
\begin{equation}\label{rectangular}
V_{r,\alpha} :=  \Span_{\mathbb F_q}   \left\{  \left(  \frac{(xy)^i M_r(xy,y)^j}{    y^{-1}Q_r(xy,y) }  \right):     
r-\alpha\leq i\leq r+h   , \,0\leq j\leq d-1 \right\}      \,.
\end{equation}
Then we have 

\begin{itemize}

\smallskip

\item[{\rm (i)}] $\Lambda_{\ell,\ell}(V_{r,\alpha}) \subset  V_{r,\frac{\alpha}{q} + 1 -\frac{1}{q}},$

\smallskip

\item[{\rm (ii)}] $V_{r,\alpha}=V_{r,0}$ if $\alpha <1$, and

\smallskip

\item[{\rm (iiii)}] $\Lambda_{\ell,\ell}(V_{r,0})\subset V_{r,0}$.
\end{itemize}
Furthermore,  if  $\ell\leq q-2$ then $\Lambda_{\ell,\ell} (V_{r,1}) \subset V_{r,0}$.
\end{lemma}

\begin{proof}
Note that (i) and (ii) implies (iii), while (ii) is trivial, so we just have to prove (i).  
Let $i$ and $j$ be two integers with 
\begin{equation}\label{eq: ij}
r-\alpha\leq i\leq r+h \;\;\mbox{ and }\;\; 0\leq j\leq d-1\,,
\end{equation}
so that 
$$
 \frac{(xy)^i M_r(xy,y)^j}{    y^{-1}Q_r(xy,y) } \in V_{r,\alpha}\,.
$$
We first infer from \eqref{eq: gq} that 
\begin{equation}\label{eq: pol}
\Lambda_{\ell,\ell}\left( \frac{(xy)^i M_r(xy,y)^j}{    y^{-1}Q_r(xy,y) }  \right) 
= \frac{\Lambda_{\ell,\ell}\left((xy)^i M_r(xy,y)^jy^{1-q}Q_r(xy,y)^{q-1}\right)}{y^{-1}Q_r(xy,y)} \, .
\end{equation}
Developing $Q_r(xy,y)^{q-1}$, we note that $ (xy)^i M_r(xy,y)^j   y^{1-q}Q_r(xy,y)^{q-1}$  is an 
$\mathbb F_q$-linear combination of elements of the form 
$y^{1-q}(xy)^{i+n} M_r(xy,y)^{j+m}$ 
where
$0\leq  n \leq (q-1)h$ and  $0\leq  m\leq d(q-1)$. 
Setting $i'=i+n$ and $j'=j+m$, we obtain
\begin{equation}\label{eq: i'j'}
r-\alpha \leq i'\leq r+qh \;\;\mbox{ and }\;\;0\leq j' \leq qd-1 \,.
\end{equation}
Recall that $\Lambda_{\ell,\ell}( x^{a}y^{b})$ is nonzero if and only if $a\equiv b \equiv \ell \bmod q$. 
Using Property \eqref{eq: gq}, we obtain that either 
$\Lambda_{\ell,\ell}(y^{1-q}(xy)^{i'} M_r(xy,y)^{j'})=0$, or 

\begin{eqnarray}\label{applying_Cartier-0}      
\Lambda_{\ell,\ell}\left(   (xy)^{i'} y^{1-q}M_r(xy,y)^{j'} \right)      &  = & 
     \Lambda_{\ell, \ell} \left( (xy)^{ i'} y^{1-q}( M_r(xy,y)^{j'\bmod q}     \right)       M_r(xy,y)^{\lfloor \frac{j'}{q}\rfloor}  
  \nonumber \\ &   {\substack{*\\=}}&  
   \Lambda_{\ell, \ell} \left( (xy)^{i'}    (xy)^{r(q-1)}     \right)      M_r(xy,y)^{\lfloor \frac{j'}{q}\rfloor}     \nonumber \\ & = &
    (xy)^{\frac{i' +r(q-1)-\ell}{q}}       M_r(xy,y)^{\lfloor \frac{j'}{q}\rfloor}, 
    \end{eqnarray}
where the asterisked equality follows because the only way this expression is nonzero is if $j'\equiv q-1\bmod q$, 
and in this case only the  $(xy)^{r(q-1)}y^{q-1}$  term  in $M_r(xy,y)^{q-1}$ will lead to a nonzero term after application of $\Lambda_{\ell,\ell}$. 
By \eqref{eq: i'j'},  we have 
\begin{equation}\label{eq: shrink2}
0\leq \lfloor j'/q\rfloor \leq d-1 \;\;\mbox{ and }\;\; r-(\alpha/q+1-1/q)\leq \frac{i' +r(q-1)-\ell}{q}\leq r+h \,,
\end{equation}
and we infer from \eqref{eq: pol} and \eqref{applying_Cartier-0} 
that $\Lambda_{\ell,\ell}\left(V_{r,\alpha} \right) \subset V_{r,\alpha/q+1-1/q}$. 
Finally, inspection of the case $\ell<q-1$ and $\alpha=1$ gives us  $\Lambda_{\ell,\ell}\left(V_{r,1} \right)\subset V_{r,0}$. 
\end{proof}
 
 We also have the following similar result.   
   
\begin{lemma}\label{shrinkage2}  
Let
$$
V^+_{r,\alpha} :=  \Span_{\mathbb F_q}   \left\{  \left(  \frac{(xy)^i M_r(xy,y)^j}{    y^{-1}Q_r(xy,y) }  \right):     
r-\alpha\leq i\leq r+h-1   , \,0\leq j\leq d-1 \right\}      \,.
$$
Then
\begin{itemize}

\smallskip

\item[{\rm (i')}] $\Lambda_{\ell,\ell}(V^+_{r,\alpha}) \subset    V^+_{r,\frac{\alpha}{q} + 1 -\frac{1}{q}}$, 

\smallskip

\item[{\rm (ii')}] $V^+_{r,\alpha}=V^+_0$ if $\alpha <1$, and

\smallskip

\item[{\rm (iii')}] $\Lambda_{\ell,\ell}(V^+_{r,0})\subset V^+_{r,0}$.  
\end{itemize}
Furthermore, if $\ell\geq 1$ then $\Lambda_{\ell,\ell}(V_{r,\alpha}) \subset  V^+_{r,\frac{\alpha}{q} + 1 -\frac{1}{q}}  $. 
\end{lemma}
 
 \begin{proof}
 The proof of (i'), (ii'), and (iii') follows the same argument as in the proof of Lemma \ref{shrinkage}. 
The fact that  $\Lambda_{\ell,\ell}(V_{r,\alpha}) \subset    V^+_{r,\frac{\alpha}{q} + 1 -\frac{1}{q}}  $ when $\ell\geq 1$ is a direct 
consequence of 
\eqref{eq: shrink2}, since $i'+r(q-1)-\ell/q<r+h$ when $\ell \geq 1$. 
 \end{proof}

\begin{lemma}\label{bound-large}
Let  $\ell\in\{0,\ldots,q-1\}$.  
Then
$\Lambda_\ell(f^{(r)}) \in \Delta(V_{r,\frac{r}{q} + 1 -\frac{1}{q}})     $. 
Furthermore, if $\ell\geq 1$, then $\Lambda_\ell(f^{(r)}) \in   \Delta(V^+_{r,\frac{r}{q} + 1 -\frac{1}{q}})$. 
\end{lemma}

%%%   
\begin{proof}
 To simplify notation we let  $V(x):=V_r(x)$, $g(x):= f^{(r)}(x)$, and $Q:=Q_r$.   
Let $M(x,y)= V(x)+x^r y$, so that $g(x)$ is a root of the polynomial $Q(x,y):=\sum_{k=0}^d A_k(x)  M(x,y)^{k}$. 
By Lemma \ref{separation-lemma}, we have 
\begin{equation}\label{eq: gdiag}
   g(x) =      \Delta \left( \frac{y\frac{ \partial Q }{\partial y}(xy,y)}
      {y^{-1}Q(xy,y)}\right)  \, \cdot
\end{equation}
Let $\ell\in\{0,\ldots,q-1\}$. 
We first infer from \eqref{eq: gq} that 
\begin{equation}\label{eq: pol2}
\Lambda_{\ell,\ell}\left( \frac{y\frac{ \partial Q }{\partial y}(xy,y)}{    y^{-1}Q(xy,y) }  \right) 
= \frac{\Lambda_{\ell,\ell}\left(y^{2-q}\frac{ \partial Q }{\partial y}(xy,y)Q(xy,y)^{q-1}\right)}{y^{-1}Q(xy,y)} \,\cdot
\end{equation}
Notice that $\frac{ \partial Q }{\partial y}(xy,y)Q(xy,y)^{q-1}$ belongs to 
$$
\Span_{\mathbb F_q}\left\{  (xy)^i M(xy,y)^j :  r\leq i\leq r+qh   , \,0\leq j\leq qd-1\right\}\,.
$$
Let $i$ and $j$ be two such integers. Equality \eqref{eq: gq} implies that 
 \begin{equation}\label{applying-Cartier-0}      
 \Lambda_{\ell,\ell}\left( y^{2-q} (xy)^{i}M(xy,y)^{j}    \right)        = 
     \Lambda_{\ell, \ell} \left(y^{2-q} (xy)^{        i  } M(xy,y)^{j\bmod q}     \right)       M(xy,y)^{\lfloor \frac{j}{q}\rfloor}  
    \end{equation}
    and thus either 
 $\Lambda_{\ell,\ell} \left(  y^{2-q}(xy)^{i}M(xy,y)^{j}     \right)=0$ 
or $ \Lambda_{\ell, \ell} \left(y^{2-q} (xy)^{  i  } M(xy,y)^{j\bmod q}     \right)\not=0$, which is only possible 
if  $j\equiv q-2 \bmod q$ or $j\equiv  q-1 \bmod q$.  
Continuing from (\ref{applying-Cartier-0}), if $j\equiv q-2 \bmod q$,
$$ 
\Lambda_{\ell,\ell}\left( y^{2-q} (xy)^{i}M(xy,y)^{j}    \right) =   (xy)^{\frac{i+(q-2)r-\ell}{q}} M(xy,y)^{\lfloor \frac{j}{q}\rfloor}\, ,  
$$ 
and if $j\equiv q-1 \bmod q$,
\begin{equation*} \Lambda_{\ell,\ell}\left( y^{2-q} (xy)^{i}M(xy,y)^{j}    \right)  =   - \Lambda_{\ell,\ell}\left(        (xy)^{i+(q-2)r}  V(xy)   \right) M(xy,y)^{\lfloor \frac{j}{q}\rfloor} \,.
\end{equation*}
 In all cases, $\Lambda_{\ell,\ell}\left( y^{2-q} (xy)^{i}M(xy,y)^{j} \right)$ belongs to 
 $$
 \Span_{\mathbb F_q}\left\{  (xy)^{i'} M(xy,y)^{j'} :  r-r/q-\ell /q \leq i'\leq r+h-\ell/q   , \,0\leq j'\leq d-1\right\}\,.
 $$
 Combining this with \eqref{eq: pol2}, we get that 
 $$
 \Lambda_{\ell,\ell}\left( \frac{y\frac{ \partial Q }{\partial y}(xy,y)}{    y^{-1}Q(xy,y) }  \right) \in V_{r,\frac{r}{q} + 1 -\frac{1}{q}}    \,.
 $$
 Furthermore if $\ell\geq 1$, we see that 
 $$
  \Lambda_{\ell,\ell}\left( \frac{y\frac{ \partial Q }{\partial y}(xy,y)}{    y^{-1}Q(xy,y) }  \right) \in V^+_{r,\frac{r}{q} + 1 -\frac{1}{q}}  \,.
 $$
By Equalities \eqref{eq: commute} and \eqref{eq: gdiag}, this ends the proof.    
\end{proof}

Before proving Theorem \ref{bounding-f-kernel}, we need a last auxiliary result, 
which gives the intertwining relations between the operators $\Lambda_\ell$ and multiplication 
by a power of $x$. Its proof is straightforward.

 \begin{lemma}\label{intertwining} Let $0\leq j \leq q-1$ and $0 \leq \ell \leq q-1$. Then 
 \begin{equation*}
 \Lambda_{\ell}(x^j f(x))=
\begin{cases}
  \Lambda_{\ell-j} (f(x))  & \text{ if } j\leq \ell\\   
x\Lambda_{q+\ell -j}(f(x))  & \text{ if } j>\ell. 
\end{cases}
\end{equation*} 

 \end{lemma}

 Note in particular that Lemma \ref{intertwining} implies that $\Lambda_0^{n}(xf(x))= x\Lambda_{q-1}^{n}(f(x))$ for each $n>0$.

 \begin{proof}[Proof of Theorem \ref{bounding-f-kernel} ]
 Recall that 
$$
f(x)=a_0+a_1x+\cdots+a_rx^r+x^rf^{(r)}(x)\,.
$$
Lemma \ref{bound-large} implies that $\Lambda_\ell(f^{(r)})\in      \Delta(V_{r,\frac{r}{q} + 1 -\frac{1}{q}}) $ for all $\ell$ and that 
$\Lambda_\ell(f^{(r)})\in   \Delta(V^+_{r,\frac{r}{q} + 1 -\frac{1}{q}}) $ if $\ell \geq 1$.
Let  
\begin{equation}\label{eq:t}
t_0=\lfloor \log_q r\rfloor  \,.
\end{equation} 
Using Properties \eqref{eq: gq} and \eqref{eq: pol0}, 
we obtain that for $t\geq t_0$, 
for any $( \ell_t, \ldots \ell_0)$, there is a unique $(\tilde  \ell_t, \ldots\tilde  \ell_0)$ and a unique $i \in \{0,1\}$ such that
$$ \Lambda_{\ell_t} \Lambda_{\ell_{t-1}}  \cdots  \Lambda_{\ell_0} (x^rf^{(r)})=  x^{i}\Lambda_{\tilde \ell_t} \Lambda_{\tilde \ell_{t-1}}  \cdots  \Lambda_{\tilde\ell_0} (f^{(r)}).$$
Iterating Lemmas \ref{shrinkage} and  \ref{shrinkage2},  we have:
\begin{itemize}
\smallskip
\item[{\rm (a)}]  If $\ell_t q^t+ \ell_{t-1}   q^{t-1}+ \cdots 
 + \ell_1 q +\ell_0> r$, then 
 $$
 \Lambda_{\ell_t} \Lambda_{\ell_{t-1}}  \cdots  \Lambda_{\ell_0} (f)    = \Lambda_{\tilde \ell_t} \Lambda_{\tilde \ell_{t-1}}  \cdots  \Lambda_{\tilde\ell_0} (f^{(r)})\in \Delta(V_{r,0}^+) \,, $$ as
 for at least one $j\geq t_0$,  we have $\tilde \ell_{j}$ is nonzero,  so that we can apply Lemma \ref{shrinkage2}.
 \vspace{1.5em}
\item[{\rm (b)}]  If $\ell_t q^t+ \ell_{t-1}   q^{t-1}+ \cdots 
 + \ell_1 q +\ell_0 = r$, then 
 $$
 \Lambda_{\ell_t} \Lambda_{\ell_{t-1}}  \cdots  \Lambda_{\ell_0} (f) = a_r + \Lambda_0^{t+1}(f^{(r)})\in \F_q+  \Delta(V_{r,0})\,.
 $$
 \smallskip
\item[{\rm (c)}]   If $\ell_t q^t+ \ell_{t-1}   q^{t-1}+ \cdots 
 + \ell_1 q +\ell_0< r$, then 
 $$
 \Lambda_{\ell_t} \Lambda_{\ell_{t-1}}  \cdots  \Lambda_{\ell_0} (f) = \alpha+x       \Lambda_{\tilde \ell_t} \Lambda_{\tilde \ell_{t-1}}  \cdots  \Lambda_{\tilde\ell_0} (f^{(r)})  
  \in  \F_q + x       \Delta(V^+_{r,1})\, ,$$
  since $\tilde \ell_{t}= p-1$ for $j\geq t_0$, so that we can apply Lemmas \ref{shrinkage} and \ref{shrinkage2}.
 \end{itemize}
 
 We stress that $\ell_t$ can be zero, so that each of cases (b) and (c) can correspond to  arbitrarily long  compositions of Cartier operators.
Notice that 
\begin{eqnarray*}  x       \Delta(V^+_{r,1}) &=& \Delta(xy V_{r,1}^+)\\& = & \Delta (V_{r,0})\, .\end{eqnarray*}
Meanwhile, $\Delta(V_{r,0}^+) \subset \Delta(V_{r,0})$, so that  the total contribution from (a), (b) and (c) combined is of dimension at most $(h+1)d+1$.

On the other hand, defining 
\begin{equation}\label{initial}\mathcal E(f)= \{  \Lambda_{\ell_t} \cdots \Lambda_{\ell_0} (f): t <   t_0    \}\, ,\end{equation}  
we claim that $\vert \mathcal E(f) \vert = o(1)q^{hd}$, where the the $o(1)$ term tends to $0$ for 
large values of any of $q,d,h$.  
Indeed, $\vert \mathcal E(f) \vert \leq q^{t_{0}}$, while  by \eqref{determinant} we have 
 $r\leq h(2d-1)$.  Hence Equality \eqref{eq:t}  gives the claim.  
We conclude that 
$$
\vert \Omega_1(f) \vert \leq  q^{(h+1)d+1}+o(q^{hd}).$$ Now \eqref{eq: Eil}  completes the proof of  Theorem \ref{bounding-f-kernel}.
\end{proof}

\begin{remark}
Note that in the smooth case $r=0$, there is no contribution from (c) above, and, since $a_0=0$, we have   $\Omega_1(f)\backslash\{f \}\subset \Delta(V_{0,0})$, so that we recover Theorem \ref{smooth-case}. Also for the smooth case, (b) is precisely the orbit of $f$ under  $\Lambda_0$, and Bridy deals with this orbit by showing that its cardinality is at most $o(q^{hd})$; this is where Bridy uses additional arguments involving  Landau's function.

Similarly, in the general case,  Cases (b) and (c) correspond to $r+1$ infinite orbits, all of whose size could be  bounded by Bridy's arguments. This should lead to 
\begin{eqnarray*}\Omega_1(f) & = &  |\Delta(V_r,0^{+})| +o(q^{hd}) \\ &  = & (1+o(1))q^{hd},\end{eqnarray*}
which is precisely Bridy's bound \eqref{eq: bridy2}.
This would have identified (a) as containing the bulk of $f$'s kernel. 
 However it is not clear how to bound the contributions from (b) and (c) in an elementary fashion.

\end{remark}

%%%%%%%%%%%%%%
\section{Bounds for the direct reading complexity}\label{direct}

A  sequence ${\bf a} = (a_n)_{n \geq 0}$ can be $q$-automatic in either {\em reverse} or {\em direct} reading. In the former case, we feed the base-$q$ expansion of $n$ into the $q$-automaton starting with the least significant digit, and in the latter, starting with the most significant digit. Now  ${\bf a}$ is $q$-automatic in direct reading if and only if it is $q$-automatic in reverse reading, but the direct and reverse reading minimal automata that generate  ${\bf a}$ are generally different and thus there are two notions of complexity; thus far we have bounded the reverse complexity $\comp_q({\bf a})$. We now wish to bound the forward reading complexity 
$\stackrel{\longrightarrow}{\comp_q}({\bf a})$, which equals the number of states in a minimal $q$-automaton generating ${\bf a}$. For the sake of symmetry, henceforth we write 
$\stackrel{\longleftarrow}{\comp_q}({\bf a})$ instead of $\comp_q({\bf a})$.

If we start with a reverse reading  $q$-automaton whose states are labelled using 
elements of an $\F_q$-vector space of dimension $k$, then  using basic duality theory  for vector spaces, 
we obtain the following result. 

\begin{proposition}\label{prop: cartier}
Let $f(x)=\sum_{n=0}^{\infty} a_nx^n\in \mathbb F_q[[x]]$. If there exists  
a $\mathbb F_q$-vector space $V$ of dimension $m$ that contains $f$ and that is invariant under the action of $\Omega_1$,  
then both $\stackrel{\longrightarrow}{\comp_q}({\bf a})$ and $\stackrel{\longleftarrow}{\comp_q}({\bf a})$ 
are bounded by $q^m$.   
\end{proposition}

Proposition \ref{prop: cartier} is 	a rephrasing of Proposition 2.4 in \cite{Bridy}, where Bridy uses the notion of a 
$q$-\emph{presentation} for the sequence ${\bf a}$. This concept is analogous to the notion of a {\em recognizable} rational series, which follows from the fact that a deterministic finite automaton with 
output in a finite field can be seen as a weighted automaton, so that  its rational series is recognizable
 \cite{Berstel-Reutenauer}. 

Bridy shows that he can inject $\Omega_1(f)$ into a vector space of dimension $(h+1)d$, so that he bounds  
$\stackrel{\longrightarrow}{\comp_q}({\bf a})$ by $q^{(h+1)d}$.  Here too our bounds are only slightly weaker, particularly 
in the smooth case.

 \begin{corollary}\label{reverse-direct-smooth}
Let $f(x)=\sum_{n\geq 0} a_n x^n \in \F_q[[x]]$ be an algebraic power series of degree $d$ and height $h$, and 
let $P(x,y)\in \F_q[x,y]$ denote the minimal polynomial of $f$. 
\begin{itemize}
\item[{\rm (i)}]
If $\frac{\partial P}{\partial y}(0,0)\neq 0$ and $f(0)=0$, 
then $\stackrel{\longrightarrow}{\comp_q}({\bf a})\leq q^{1+(h+1)d}$.
\item[{\rm (ii)}]
Otherwise, let $r$  be  the order at $0$ of the resultant of $P(x,y)$ 
and $ \frac{\partial P}{\partial y}(x,y)$. Then
 $\stackrel{\longrightarrow}{\comp_q}({\bf a})\leq q^{1+(h+1)d +r}$. 
 In particular, $\stackrel{\longrightarrow}{\comp_q}({\bf a})\leq q^{(3h+1)d -h+1}$.
\end{itemize}
 \end{corollary}

\begin{proof}

We first prove (i), which corresponds to the smooth case.    
The proof of Theorem \ref{smooth-case} tells us that we can realise most of $\Omega_1(f)$ in the $\F_q$-vector space  $\Delta(W)$ where $W$ is defined in (\ref{nested-spaces}), of dimension $(h+1)d$. The only element of $\Omega_1(f)\backslash W$ is $f$ itself. 
Thus we add $f$ to $\Delta(W)$, to obtain a vector space of dimension $1+(h+1)d$ into which we have injected $\Omega_1(f)$.
Applying Proposition \ref{prop: cartier} gives the result. 

In the general case, an inspection of the last part of the proof of Theorem \ref{bounding-f-kernel} tells us that the vector space $\Delta(V_{r,0} )$, defined in Lemma \ref{shrinkage} contains most of $\Omega_1(f)$. Here, we have to add another basis element, representing the addition of constants in the expression $\F_q+\Delta(V_{r,0} )$. We must also add (at most) $r$ basis elements, to represent the set $\mathcal E(f)$ defined in (\ref{initial}). The result follows form Proposition \ref{prop: cartier} and the bound in  \eqref{determinant}.
\end{proof}

 %%%%%%%%%%%%%%%%%%%%%%%%%%%  
\section{Linking the genus to our setting}\label{genus}
     
  Let  $P(x,y)\in \F_q[x,y] $,   let $\mathcal P$ be the Newton polygon of $P(x,y)$, and let $g_P$ be the number of integral points in the interior of $\mathcal P$.  We recall that $g_P$ is closely related to the genus $g$ of the curve associated with $P$. In a majority of cases we have $g_P=g$, and in general
  $g\leq g_P$. For example, 
    if $P(x,y)$ is  irreducible over every algebraic extension of $\F_q$, then
 this follows by work of Beelen \cite{Beelen-2009}.
 
 If 
    $f(x)  \in \F_q[[x]]$ is a  root of $P(x,y) \in \F_q[x,y]$ of degree $d$ and height $h$, then $g_P\leq (h-1)(d-1)$. Bridy's bound 
    \eqref{eq: bridy}  is better than his general bound \eqref{eq: bridy2} when the genus is smaller than $(h-1)(d-1)$. In our case too, if  $g_P<(h-1)(d-1)$, then it is possible to obtain better bounds: we can choose smaller vector spaces to work with, taking into account the shape of $\mathcal P$, instead of the full rectangular bases that we have in   \eqref{nested-spaces}   or \eqref{rectangular}.

 \section*{Acknowledgement}
The second author is grateful to Peter Beelen for a discussion of his work. She also thanks  IRIF, Universit\'e Paris Diderot-Paris 7, for its hospitality and support.
                 
{\footnotesize
\bibliographystyle{alpha}
\bibliography{bibliography}

\def\ocirc#1{\ifmmode\setbox0=\hbox{$#1$}\dimen0=\ht0 \advance\dimen0
  by1pt\rlap{\hbox to\wd0{\hss\raise\dimen0
  \hbox{\hskip.2em$\scriptscriptstyle\circ$}\hss}}#1\else {\accent"17 #1}\fi}
\begin{thebibliography}{CKMFR80}

\bibitem[AB12]{AB-1}
Boris Adamczewski and Jason~P. Bell.
\newblock On vanishing coefficients of algebraic power series over fields of
  positive characteristic.
\newblock {\em Invent. Math.}, 187(2):343--393, 2012.

\bibitem[AB13]{Adam-Bell-2013}
Boris Adamczewski and Jason~P. Bell.
\newblock Diagonalization and rationalization of algebraic {L}aurent series.
\newblock {\em Ann. Sci. \'Ec. Norm. Sup\'er. (4)}, 46(6):963--1004, 2013.

\bibitem[AS03]{Allouche-Shallit}
Jean-Paul Allouche and Jeffrey Shallit.
\newblock {\em Automatic Sequences: Theory, Applications, Generalizations}.
\newblock Cambridge University Press, Cambridge, 2003.

\bibitem[BCCD19]{Bostan-Caruso-Christol-Dumas}
Alin Bostan, Xavier Caruso, Gilles Christol, and Philippe Dumas.
\newblock Fast coefficient computation for algebraic power series in positive
  characteristic.
\newblock volume~2, pages 119--135, 2019.

\bibitem[Bee09]{Beelen-2009}
Peter Beelen.
\newblock A generalization of {B}aker's theorem.
\newblock {\em Finite Fields Appl.}, 15(5):558--568, 2009.

\bibitem[BR11]{Berstel-Reutenauer}
Jean Berstel and Christophe Reutenauer.
\newblock {\em Noncommutative rational series with applications}, volume 137 of
  {\em Encyclopedia of Mathematics and its Applications}.
\newblock Cambridge University Press, Cambridge, 2011.

\bibitem[Bri17]{Bridy}
Andrew Bridy.
\newblock Automatic sequences and curves over finite fields.
\newblock {\em Algebra Number Theory}, 11(3):685--712, 2017.

\bibitem[Chr79]{Christol-1979}
Gilles Christol.
\newblock Ensembles presque periodiques {$k$}-reconnaissables.
\newblock {\em Theoret. Comput. Sci.}, 9(1):141--145, 1979.

\bibitem[CKMFR80]{CKMR}
Gilles Christol, Teturo Kamae, Michel Mend{\`e}s~France, and G\'erard Rauzy.
\newblock Suites alg\'ebriques, automates et substitutions.
\newblock {\em Bull. Soc. Math. France}, 108(4):401--419, 1980.

\bibitem[DL87]{DL}
J.~Denef and L.~Lipshitz.
\newblock Algebraic power series and diagonals.
\newblock {\em J. Number Theory}, 26(1):46--67, 1987.

\bibitem[Eil74]{Eilenberg}
Samuel Eilenberg.
\newblock {\em Automata, languages, and machines. {V}ol. {A}}.
\newblock Academic Press [A subsidiary of Harcourt Brace Jovanovich,
  Publishers], New York, 1974.
\newblock Pure and Applied Mathematics, Vol. 58.

\bibitem[FKdM00]{FKM}
Jean Fresnel, Michel Koskas, and Bernard de~Mathan.
\newblock Automata and transcendence in positive characteristic.
\newblock {\em J. Number Theory}, 80(1):1--24, 2000.

\bibitem[Fur67]{Furstenberg}
Harry Furstenberg.
\newblock Algebraic functions over finite fields.
\newblock {\em J. Algebra}, 7:271--277, 1967.

\bibitem[Har88]{H}
Takashi Harase.
\newblock Algebraic elements in formal power series rings.
\newblock {\em Israel J. Math.}, 63(3):281--288, 1988.

\bibitem[Har89]{HII}
Takashi Harase.
\newblock Algebraic elements in formal power series rings. {II}.
\newblock {\em Israel J. Math.}, 67(1):62--66, 1989.

\bibitem[RY15]{Rowland-Yassawi-2015}
Eric Rowland and Reem Yassawi.
\newblock Automatic congruences for diagonals of rational functions.
\newblock {\em J. Th\'eor. Nombres Bordeaux}, 27(1):245--288, 2015.

\bibitem[Spe]{Speyer-2010}
Christol's theorem and the {C}artier operator.
\newblock
  \url{https://sbseminar.wordpress.com/2010/02/11/christols-theorem-and-the-cartier-operator/}.
\newblock Accessed: 2010-09-30.

\bibitem[SW88]{SW}
Habib Sharif and Christopher~F. Woodcock.
\newblock Algebraic functions over a field of positive characteristic and
  {H}adamard products.
\newblock {\em J. London Math. Soc. (2)}, 37(3):395--403, 1988.

\end{thebibliography}
}

 \end{document}